\documentclass[12pt]{article}
\usepackage{amsmath,amsfonts,amssymb,amsthm}

\newtheorem{theorem}{Theorem}
\newtheorem{remark}{Remark}
\newtheorem{lemma}{Lemma}

\newtheorem{exam}{Example}

\newcommand{\E}{\mathsf{E}}
\newcommand{\Prob}{\mathsf{P}}

\begin{document}

\begin{center}
{\Large\bf On the governing equations for Poisson and Skellam
processes time-changed by inverse subordinators}
\end{center}

\begin{center}{{K. V. Buchak}\footnote{Department of Probability Theory, Statistics and Actuarial Mathematics,
        Taras Shevchenko National University of Kyiv,
        Vo\-lo\-dy\-myr\-ska 64, 01601, Kyiv, Ukraine,
        E-mail:{ kristina.kobilich@gmail.com}}
 and {L. M. Sakhno}\footnote{Department of Probability Theory, Statistics and Actuarial Mathematics,
        Taras Shevchenko National University of Kyiv,
        Vo\-lo\-dy\-myr\-ska 64, 01601, Kyiv,
        Ukraine, E-mail:{ lms@univ.kiev.ua}}}

\end{center}

%\medskip

\begin{abstract}
\noindent    In the paper we present the governing equations for
marginal distributions of Poisson and Skellam processes
time-changed by inverse subordinators. The equations are given in
terms of convolution-type derivatives.
\end{abstract}

%%%%%%%%%%%%%%%%%%%%%%%%%%%%%%%%%%%%%%%%%%%%%%%%

\section{Introduction}
Time-changed Poisson processes $N(H^f(t))$, $t\ge0$, where
$H^f(t)$ is a subordinator with the Laplace exponent $f$,
independent of $N(t)$, provide a rich class of tractable and
flexible models with applications in various applied areas. The
processes $N(H^f(t))$, $t\ge0$, have positive integer-valued jumps
whose distribution can be expressed in terms of the corresponding
Bern\v{s}tein function $f$. In the papers \cite{OT}, \cite{GOS}
the distributional properties, hitting times and governing
equations for such processes were derived and specified for
several choices of $H^f(t)$ (some particular cases can be also
found  in \cite{OP}, \cite{BS}, \cite{KS}, \cite{KNV}).

Particular attention in literature has been gained by the
processes $N(S_\alpha(t))$, where $S_\alpha(t)$ is a stable
subordinator, with the Bern\v{s}tein function $f(s)=s^\alpha$,
$0<\alpha<1$. In this case the time-changed Poisson process is
called  a space-fractional Poisson process. The study of these
processes was undertaken, e.g., in the papers \cite{OP},
\cite{OT}, \cite{GOS}, to mention only few.

Another interesting time-changed model is provided by the process
$N(Y_\alpha(t))$, where $Y_\alpha(t)$ is the inverse process for
the stable subordinator $S_\alpha(t)$. Such a process is called a
time-fractional Poisson process. The important feature of this
process is that the probabilities $p_k^\alpha(t)=\Prob
\left\{N\left(Y_\alpha(t)\right)=k\right\}$ obey the fractional
equation of the form:
\begin{equation}\label{intr1}
\mathcal{D}_t^{\alpha}p_k^{\alpha}(t)=-\lambda\left(p_k^{\alpha}(t)-p_{k-1}^{\alpha}(t)\right),\quad k=0,1,2,\dots,
\end{equation}
with appropriate initial conditions, where $\mathcal{D}_t^{\alpha}$ is the fractional Caputo-Djrbashian derivative, $\lambda$ is the intensity parameter of the Poisson process $N(t)$ (see, for example, Beghin and Orsingher \cite{BO09}, \cite{BO10}).

We have supposed so far that the Poisson process $N(t)$ is homogeneous, and, correspondingly, $N(Y_\alpha(t))$ is a fractional (in time) homogeneous Poisson process. In the paper by Leonenko, Scalas, Trinh \cite{Leon} the non-homogeneous fractional Poisson process was introduced and studied, namely, the time-changed process $N(Y_\alpha(t))$, where $N(t)$ is supposed to be non-homogeneous. It was shown that the  marginal  distributions of such process satisfy the fractional difference-integral equations, involving the fractional Caputo-Djrbashian derivative, these governing equations generalize the equations \eqref{intr1} which hold in the homogeneous case.

In the recent paper by Kochubei \cite{K} and later in the paper by
Toaldo \cite{T}, the new types of differential operators are
presented which are related to Bern\v{s}tein functions and
generalize the classical Caputo-Djrbashian and Riemann-Liouville
fractional derivatives. In the paper \cite{K} these operators are
called differential-convolution operators, and in \cite{T} they
are called convolution-type derivatives with respect to
Bern\v{s}tein functions. It is shown in \cite{T} that these
derivatives provide the unifying framework for the study of
subordinators and their inverse processes, and, in particular, the
governing equations  for densities of subordinators and their
inverses are obtained in terms of the convolution-type
derivatives. The introduction of these derivatives has also
inspired numerous recent studies of new types of equations
suitable to describe anomalous diffusion and other complex
processes.

In the present paper we study the time-changed processes
\begin{equation}N(Y^f(t)), \,\,t\ge 0, \,\, \mbox{\rm and} \,\, S(Y^f(t)), \,\,t\ge 0,\label{intr2}
\end{equation}
where $N(t)$ is a (non-homogeneous) Poisson process, $S(t)$ is a
Skellam process, and $Y^f(t)$ is an inverse subordinator. Note
that with upper index $f$ we refer to the Bern\v{s}tein function
of the subordinator to which $Y^f(t)$ is the inverse.

We obtain the governing equations for marginal distributions of
the processes \eqref{intr2} which generalize the known results,
where the time change is performed by means of the inverse stable
subordinator (see, \cite{Leon}, \cite{KLS}). This generalization
is done by the use of appropriate convolution-type derivatives  in
lieu of fractional ones.

We note that governing equations for marginal distributions of the
process $N(H^f(t))$, $t\ge0$, with $H^f(t)$ being an arbitrary
subordinator are presented in \cite{OT}. Namely, these are
difference-differential equations (form of which depends on the
corresponding Bern\v{s}tein function $f$),  and these equations
can be specified for particular models (see, e.g., \cite{GOS},
\cite{OT}, \cite{BS}, \cite{KNV}). For some models more general
fractional difference-differential equations can be written. We
refer, for example, to \cite{OP}, and to the very recent paper
\cite{ALM}, where new properties of fractional Poisson processes,
mixed-fractional Poisson processes and  fractional Poisson fields
are presented and the corresponding fractional differential
equations are studied.

The paper is organized as follows. In Section 2 we give all
necessary definitions and results from \cite{T} on convolution
type-derivatives which will be used in the next sections. In
Section 3 we present the governing equations for the marginal
distribution and moment generating function for the Poisson
process time-changed by a general inverse subordinator $Y^f(t)$.
The equations involve the Caputo-Djrbashian convolution-type
derivatives. As a direct corollary of the theorem on governing
equations for time-changed Poisson process, we obtain that the
Laplace transform of the inverse subordinator $Y^f(t)$ is an
eigenfunction of the generalized Caputo-Djrbashian
convolution-type derivative w.r.t. the corresponding Bern\v{s}tein
function $f$. In Section 4 the analogous results are derived for
time-changed Skellam processes.

\section{Preliminaries}

Let $f(x)$ by a Bern\v{s}tein function:
\begin{equation}\label{fx}
f(x)=a+bx+\int_{0}^{\infty}\left(1-e^{-xs}\right)\overline{\nu}(ds), \qquad x>0, a,b\geq0,
\end{equation}
$\overline{\nu}(ds)$ is a non-negative measure on $(0,\infty)$, referred to as the L\'evy measure for $f(x)$, such that
\begin{equation*}
\int_{0}^{\infty}\left(s\wedge 1\right)\overline{\nu}(ds)<\infty.
\end{equation*}
In the papers \cite{K}, \cite{T} the convolution-type derivatives with respect to Bern\u{s}tein functions were introduced, which generalize the classical fractional derivatives.

We present the definitions and some facts following  \cite{T}.

To write the governing equations for time-changed processes, we will use  the generalized Caputo-Djrbashian (C-D) derivative w.r.t. the Bern\v{s}tein function $f$, which is defined on the space of absolutely continuous functions as follows (see \cite{T}, Definition 2.4):
\begin{equation}\label{fDt}
^f{\mathcal D}_t u(t)=b \frac{d}{dt}u(t)+\int_{0}^{t}\frac{\partial }{\partial t }u(t-s)\nu (s)ds,
\end{equation}
where $\nu (s)=a+\overline{\nu}(s,\infty)$ is the tail of the L\'evy measure $\overline{\nu}(s)$ of the function $f$.

In the case when $f(x)=x^{\alpha},x>0, \alpha\in (0,1)$, the derivative \eqref{fDt} becomes:
\begin{equation*}
^f\mathcal{D}_t u(t)=\mathcal{D}_t^{\alpha}u(t),
\end{equation*}
where $\mathcal{D}_t^{\alpha}u(t)$ is the fractional C-D
derivative:
\begin{equation*}
\mathcal{D}_t^{\alpha}u(t)=\frac{d^{\alpha}}{dt^{\alpha}}u(t)=\frac{1}{\Gamma(1-\alpha)}\int_{0}^{t}\frac{u^\prime(s)}{\left(t-s\right)^{\alpha}}ds
\end{equation*}
(see Remarks 2.6 and 2.5 in \cite{T}).

According to Lemma 2.5 \cite{T}, the following relation holds for the Laplace transform of the derivative \eqref{fDt}:
\begin{equation}\label{LfDt}
\mathcal{L}\left[^f\mathcal{D}_t u(t)\right](s)=f(s)\mathcal{L}\left[u(t)\right](s)-\frac{f(s)}{s}u(0), s>s_0,
\end{equation}
for $u(t)$ such that $|u(t)|\leq\mathsf{M} e^{s_0t}$, $M$ and $s_0$ are some constants.

The generalization of the classical Riemann-Liouville (R-L) fractional derivative is introduced in \cite{T} by means of another convolution-type derivative w.r.t. $f$, which is given by the following formula:
\begin{equation}\label{fDDt}
^f\mathbb{D}_t u(t)=b \frac{d}{dt}u(t)+\frac{d}{dt}\int_{0}^{t}u(t-s)\nu (s)ds
\end{equation}
(see, \cite{T}, Definition 2.1).

The derivatives $^f\mathcal{D}_t$ and $^f\mathbb{D}_t$ are related as follows:
\begin{equation}\label{fDtrelation}
^f\mathbb{D}_t u(t)=^f\mathcal{D}_t+\nu (t)u(0)
\end{equation}
(see, \cite{T}, Proposition 2.7).

Let $H^f(t), t\geq0,$ be a subordinator with the Laplace exponent $f$ given by \eqref{fx}, and $Y^f(t)$ be its inverse process defined as
\begin{equation}\label{Yf}
Y^f(t)=inf \left\{s\geq0:H^f(s)>t\right\}.
\end{equation}
It was shown in \cite{T} that the distribution of the inverse subordinator $Y^f(t)$ has a density $l^f(t,s)$ and its Laplace transform with respect to $t$ has the form:
\begin{equation*}
\mathcal{L}_t\left(l^f(t,s)\right)(r)=\frac{f(r)}{r}e^{-sf(r)},
\end{equation*}
provided that the following condition holds:

\medskip
\noindent   {\bf Condition  I.}
$\overline{\nu}(0,\infty)=\infty$ and the tail $\nu(s)=a+\overline{\nu}(s,\infty)$ is absolutely continuous.
\medskip

It was shown in \cite{T} that the convolution-type derivatives $^f\mathcal{D}_t$ and $^f\mathbb{D}_t$ are useful tools which allow to study  the properties of subordinators and their inverses in the unifying manner. In particular, the governing equations for their densities can be written down in terms of convolution-type derivatives. According to Theorem 4.1 \cite{T},
the density $l^f(t,u)$ of the inverse subordinator $Y^f(t)$ satisfies the following equation:
\begin{equation}\label{dlf}
^f{\mathbb{D}}_t l^f(t,u)=-\frac{\partial}{\partial  u} l^f(t,u),
\end{equation}
subject to
\begin{equation}\label{dlfin}
l^f(t, u/b)=0,\,\,\,\,l^f(t,0)=\nu(t), \,\,\,\, l^f(0,u)=\delta(u).
\end{equation}

In the paper \cite{K} the solution to the Cauchy problem
 for the equations involving differential-convolution operators were presented,
 and the connections with Poisson processes time-changed by inverse subordinators were provided.
 We will return in more details to these results in Remark 3 in Section 3.

In the present paper we will use the generalized C-D and R-L convolution-type derivatives to obtain the governing equations for time-changed Poisson and Skellam processes.

Consider a non-homogeneous Poisson process $N(t), t\geq0$, with intensity function $\lambda(t):[0,\infty)\to[0,\infty)$. Denote
\begin{equation*}
\Lambda(s,t)=\int_{s}^{t}\lambda(u)du, \qquad\Lambda(t)=\Lambda(0,t).
\end{equation*}
$N(t)$ has independent but not necessarily stationary increments and for $0\leq v<t$
\begin{eqnarray*}
    p_x(t,v)&=&\Prob\left\{N(t+v)-N(v)=x\right\}=\frac{e^{-\left(\Lambda(t+v)-\Lambda(v)\right)}\left(\Lambda(t+v)-\Lambda(v)\right)^x}{x!}\\
    &=&\frac{e^{-\left(\Lambda(v,t+v)\right)}\left(\Lambda(v,t+v)\right)^x}{x!},\qquad x=0,1,2,\dots.
\end{eqnarray*}
The distribution $p_x(t,v)$ satisfy the difference-differential equation
\begin{equation*}
\frac{d}{dt}p_x(t,v)=-\lambda(t+v)\left(p_x(t,v)-p_{x-1}(t,v\right),    \qquad  x=0,1,2,\dots.
\end{equation*}
with the usual initial conditions
\begin{equation*}
p_x(0,v)=
\begin{cases}
1,\quad x=0,\\
0,\quad x\geq1,\\
\end{cases}
\end{equation*}
$p_{-1}(t,v)=0$ (see, e.g. \cite{Leon}). We denote $p_x(t)=p_x(t,0)$.

The non-homogeneous Poisson process $N(t)$ can be represented as $N(t)
=N_1(\Lambda(t))$, where $N_1(t)$ is the homogeneous Poisson process with intensity $1$.

We will study the non-homogeneous Poisson process time-changed by an inverse subordinator $Y^f(t)$:
\begin{equation}\label{**}
N^f(t)=N(Y^f(t)), t\geq0,
\end{equation}
$f$ is the Bern\u{s}tein function of the form \eqref{fx} and we will suppose in what follows that $a=b=0$.

In the paper \cite{Leon} the authors study the case
$f(s)=s^{\alpha}$, which corresponds to the stable subordinator
and $Y^f(t)$ is the inverse stable subordinator $Y_{\alpha}(t)$.
The corresponding time-changed process \eqref{**}  is called a
fractional non-homogeneous Poisson process, for marginal
distributions of which the fractional difference-differential
equations are derived in \cite{Leon}.

In the next section, we are aimed to obtain the governing equation for marginal distributions of \eqref{**} in the general case. We show that this can be done with the use of convolution-type derivative \eqref{fDt}, applying the same line of reasonings as those used in \cite{Leon} for a particular case of the inverse stable subordinator.

\section{Poisson processes time-changed by inverse subordinators}

Consider a non-homogeneous Poisson process $N(t), t\geq0$, with intensity function $\lambda(t)$, and its increments defined for $v\geq0$ as $$I(t,v)=N(t+v)-N(v), t\geq0.$$

Taking time-change by the inverse subordinator $Y^f(t)$, we consider the process
$$N^f(t)=N\left(Y^f(t)\right), \,\,t\geq0,$$
and, for $v\geq0$, the time-changed increments process $$ I^f(t,v)=I\left(Y^f(t),v\right)=N\left(Y^f(t)+v\right)-N(v), \,\, t\geq0.$$
Their marginal distributions can be written down as follows:
\begin{eqnarray*}
    p_x^f(t,v)&=&\Prob\left\{I^f(t,v)=x\right\}=\Prob\left\{N\left(Y^f(t)+v\right)-N(v)=x\right\}
    =\int_{0}^{\infty}p_x(u,v)l^f(t,u)du\\
    &=&\int_{0}^{\infty}\frac{e^{-\Lambda(v,u+v)}\Lambda(v,u+v)^x}{x!}l^f(t,u)du, \qquad x=0,1,2,\dots
\end{eqnarray*}
and
\begin{eqnarray*}
    p_x^f(t)&=&p_x^f(t,0)=\Prob\left\{N\left(Y^f(t)\right)=x\right\}
    =\int_{0}^{\infty}p_x(u)l^f(t,u)du\\
    &=&\int_{0}^{\infty}\frac{e^{-\Lambda(u)}\Lambda(u)^x}{x!}l^f(t,u)du, \qquad x=0,1,2,\dots,
\end{eqnarray*}
where $l^f(t,u)$ is the density of the process $Y^f(t)$, which exists under Condition {\bf I}.

In the next theorem we present the governing equation for the marginal distributions $p_x^f(t,v)$.
\begin{theorem}
    Let Condition I  hold. Then the marginal distributions  $p_x^f(t,v)$ of the process $I^f(t,v)$ satisfy the following differential-integral equations
    \begin{equation*}
    ^f\mathcal{D}_tp_x^f(t,v)=\int_{0}^{\infty}\lambda(u+v)\left[-p_x(u,v)+p_{x-1}(u,v)\right]l^f(t,u)du,
    \end{equation*}
    with initial condition

    $\qquad\qquad\qquad
    p_x(0,v)=
    \begin{cases}
    1,\quad x=0,\\
    0,\quad x\geq1,\\
    \end{cases}
    $

    \noindent and $p_{-1}(0,v)\equiv0$, where $^f\mathcal{D}_t$ is the C-D convolution-type derivative with respect to function $f$ defined by \eqref{fDt}.
\end{theorem}

\begin{proof} We follow the same arguments as those used in \cite{Leon}.
    Taking the characteristic function of $p_x^f$ and the Laplace transform w.r.t. $t$ (which we denote as  $\widehat{u}$  and  $\tilde{u}$ correspondingly), we obtain:
    \begin{eqnarray*}
        \tilde{\hat{p}}_y^f(r,v)&=&\int_{0}^{\infty}\hat{p}_y(u,v)\tilde{l}^f(r,u)du=\\
        &=&\int_{0}^{\infty} \exp\left\{\Lambda\left(v,u+v\right)\left(e^{iy}-1\right)\right\}\frac{f(r)}{r}e^{-uf(r)}du\\
        &=&\frac{f(r)}{r}\left[\exp\left\{\Lambda\left(v,u+v\right)\left(e^{iy}-1\right)\right\}\left.\left[-\frac{1}{f(r)}e^{-uf(r)}\right]\right|_{u=0}^{\infty}\right.\\
        &&\left.+\frac{1}{f(r)}\int_{0}^{\infty}\left(\frac{d}{du}\Lambda(v,u+v)\right)\left(e^{iy}-1\right)\exp\left\{\Lambda\left(v,u+v\right)\left(e^{iy}-1\right)\right\}e^{-uf(r)}du\right]\\
        &=&\frac{1}{f(r)}\left[\frac{f(r)}{r}+\left(e^{iy}-1\right)\int_{0}^{\infty}\lambda(u+v)\exp\left\{\Lambda\left(v,u+v\right)\left(e^{iy}-1\right)\right\}\frac{f(r)}{r}e^{-uf(r)}du\right];
    \end{eqnarray*}
    and, therefore,
    $$
    f(r)\tilde{\hat{p}}_y^f(r,v)-\frac{f(r)}{r}=\left(e^{iy}-1\right)\int_{0}^{\infty}\lambda(u+v)\exp\left\{\Lambda\left(v,u+v\right)\left(e^{iy}-1\right)\right\}\frac{f(r)}{r}e^{-uf(r)}du.
    $$
    The left-hand side is equal to the following:
    \begin{equation*}
    f(r)\mathcal{L}_t\left\{\hat{p}_y^f(t,v)\right\}(r)-\frac{f(r)}{r}\hat{p}_y^f(0,v)=\mathcal{L}_t\left\{^f\mathcal{D}_t\hat{p}_y^f(t,v)\right\}(r),
    \end{equation*}
    in view of \eqref{LfDt} and taking into account that under conditions of the  theorem $\hat{p}_y^f(0,v)=1$.
    Therefore, inverting the Laplace transform, we have:
    \begin{equation*}
    ^f\mathcal{D}_t\hat{p}_y^f(t,v)=\left(e^{iy}-1\right)\int_{0}^{\infty}\lambda(u+v)\hat{p}_y(u,v)l^f(t,u)du,
    \end{equation*}
    and then,  inverting the characteristic function, we obtain:
    \begin{equation*}
    ^f\mathcal{D}_t p_x^f(t,v)=\int_{0}^{\infty}\lambda(u+v)\left[-p_x(u,v)+p_{x-1}(u,v)\right]l^f(t,u)du.
    \end{equation*}

    Note that equivalently the proof can be done by considering the double Laplace transform of $p_x^f(t,v)$ with respect to $x$ and $t$.
\end{proof}

From Theorem 1, taking $v=0$, we immediately obtain the corresponding results for the marginal distributions of the process $N\left(Y^f(t)\right)$ itself, which we formulate in the next two theorems.

\begin{theorem}
    Let $N\left(Y^f(t)\right)$, $t\geq0$, be a nonhomogeneous Poisson process time-changed by the inverse subordinator $Y^f(t)$, and Condition I hold. Then the marginal distributions $p_x^f(t)=\Prob \left\{N\left(Y^f(t)\right)=x\right\}, x=0,1,\dots,$ satisfy the differential-integral equations

    \begin{equation}\label{fDp}
    ^f\mathcal{D}_tp_x^f(t)=\int_{0}^{\infty}\lambda(u)\left[-p_x^f(u)+p_{x-1}^f(u)\right]l^f(t,u)du,
    \end{equation}
    with initial condition
    \begin{equation*}
    p_x(0)=
    \begin{cases}
    1,\quad x=0,\\
    0,\quad x\geq1,\\
    \end{cases}
    \end{equation*}
    $p_{-1}^f(0)=0$, where $^f\mathcal{D}_t$ is the generalized C-D derivative with respect to $f$ defined by \eqref{fDt}, and $l^f(t,u)$ is the density of the inverse subordinator $Y^f(t)$.
\end{theorem}

\begin{theorem}
    Under the conditions of Theorem 2, suppose that the Poisson process $N$ is homogeneous  with intensity $\lambda$. Then the marginal distributions $p_x^f(t)=\Prob \left\{N\left(Y^f(t)\right)=x\right\}, x=0,1,\dots,$ satisfy the differential equations:
    \begin{equation}\label{fDpx}
    ^f\mathcal{D}_tp_x^f(t)=-\lambda\left[p_x^f(t)-p_{x-1}^f(t)\right],
    \end{equation}
    with the same initial conditions as in Theorem 2.
\end{theorem}

Although Theorem 3 is a direct corollary of Theorem 1, we present now its another proof, which will be instructive for our consideration of time-changed Skellam processes in the next section.

\begin{proof}[Proof of Theorem 3] In the homogeneous case, for the probabilities $p_x^f(t)$ we have:
    \begin{equation*}
    p_x^f(t)=\Prob\left\{N\left(Y^f(t)\right)=x\right\}
    =\int_{0}^{\infty}p_x(u)l^f(t,u)du
    =\int_{0}^{\infty}\frac{e^{-\lambda u}(\lambda u)^x}{x!}l^f(t,u)du, \quad x=0,1,2,\dots
    \end{equation*}
    We take the generalized R-L convolution-type derivative $^f\mathbb{D}_t$ given by \eqref{fDDt}  and use the fact that the density $l^f(t,u)$ of the inverse subordinator satisfies the following equation (see \cite{T}):
    \begin{equation}
    ^f{\mathbb{D}}_t l^f(t,u)=-\frac{\partial}{\partial  u} l^f(t,u),
    \end{equation}
    and
    \begin{equation}
    l^f(t,0)=\nu(t), \,\,\,\, l^f(0,u)=\delta(u).
    \end{equation}

    We obtain:
    \begin{eqnarray}
    ^f{\mathbb{D}}_t p_x^f(t)
    &=&\int_{0}^{\infty}p_x(u)^f{\mathbb{D}}_tl^f(t,u)du=- \int_{0}^{\infty}p_x(u)\frac{\partial}{\partial  u}l^f(t,u)du\nonumber\\
    &=&\int_{0}^{\infty}l^f(t,u)\frac{\partial}{\partial  u}p_x(u)du-p_x(u)l^f(t,u)\big|_{u=0}^\infty\nonumber\\
    &=&\int_{0}^{\infty}l^f(t,u)(-\lambda[p_x(u)-p_{x-1}(u)])du+p_x(0)l^f(t,0)\nonumber\\
    &=&-\lambda\left[p^f_x(u)-p^f_{x-1}(u)\right]du+p_x(0)\nu(t). \label{11}
    \end{eqnarray}
    Using the relation \eqref{fDtrelation} between the convolution derivatives of C-D and R-L types, we have:
    \begin{equation}\label{12}
    ^f\mathcal{D}_t p_x^f(t) = \, ^f\mathbb{D}_t p_x^f(t)-\nu(t)p_x^f(0),
    \end{equation}
    and we note that
    \begin{equation}\label{13}
    p_x^f(0)=\int_{0}^{\infty}p_x(u)l^f(0,u)du=\int_{0}^{\infty}p_x(u)\delta(u)du=p_x(0)=1.
    \end{equation}
    From \eqref{11}, taking into account \eqref{12}-\eqref{13}, we finally obtain:
    \begin{equation*}
    ^f\mathcal{D}_tp_x^f(t)=-\lambda\left[p_x^f(t)-p_{x-1}^f(t)\right].
    \end{equation*}
\end{proof}

\begin{exam}{ In the case when $Y^f(t)=Y_{\alpha}(t), t\geq0$, is the inverse stable subordinator, that is, $f(x)=x^{\alpha},\alpha\in (0,1)$, we obtain  the known equation involving the fractional C-D derivative  $D_t^{\alpha}$  in the left-hand side:}
    \begin{enumerate}
        \item[(i)]
        if $N(t)$ is a nonhomegeneons Poisson process with intensity $\lambda$, then
        $N\left(Y_{\alpha}(t)\right)$ is a fractional nonhomogeneous Poisson process, and the governing equation \eqref{fDp} becomes
        \begin{equation*}
        \mathcal{D}_t^{\alpha}p_x^{\alpha}(t)=\int_{0}^{\infty}\lambda(u)\left[-p_x^{\alpha}(u)+p_{x-1}^{\alpha}(u)\right]h^{\alpha}(t,u)du,
        \end{equation*}
        where $h_{\alpha}(t,u)$ is the density of the inverse stable subordinator (this result was stated in \cite{Leon});

        \item[{(ii)}]
        if $N(t)$ is a homogeneous Poisson process with intensity $\lambda$, then we obtain the well known equation for probabilities of time-fractional Poisson process $N\left(Y_{\alpha}(t)\right)$:
        \begin{equation*}
        \mathcal{D}_t^{\alpha}p_x^{\alpha}(t)=-\lambda\left(p_x^{\alpha}(t)-p_{x-1}^{\alpha}(t)\right),\quad x=0,1,2,\dots,
        \end{equation*}
        (see, for example, \cite{BO09}, \cite{BO10}).
    \end{enumerate}
\end{exam}
\begin{remark}
    Moments and covariance of the nonhomogeneous
time-changed process $N\left(Y^f(t)\right)$ can be obtained with
the same arguments as in \cite{Leon}, Section 4. In fact, the
reasoning therein holds true if      instead
    of the inverse stable subordinator $Y_{\alpha}(t)$ we take a general inverse subordinator $Y^f(t)$. We come to the following formulas:
    \begin{equation*}
    \E \left[N\left(Y^f(t)\right)^k\right]=\int_{0}^{\infty}\sum_{i=1}^{k} \Lambda(x)^i  S(k,i)
    l^f(t,x)dx= \E\sum_{i=1}^{k} \Lambda\left(Y^f(t)\right)^iS(k,i)
    \end{equation*}
    where $ S(k,i)$ are the Stirling numbers of the second kind:
    $$S(k,i) =\frac{1}{i!}\sum_{j=0}^{i}(-1)^{i-j}\binom{i}{j}j^k;$$
    \begin{equation*}
    \E \left[N\left(Y^f(t)\right)\right]=\E\left[\Lambda\left(Y^f(t)\right)\right];
    \end{equation*}
    \begin{equation*}
    \E\left[\left[N\left(Y^f(t)\right)\right]^2\right]=\E\left[\Lambda\left(Y^f(t)\right)\right]+\E\left[\Lambda\left(Y^f(t)\right)^2\right].
    \end{equation*}

    \begin{equation*}
    Var\left[N\left(Y^f(t)\right)\right]=\E\left[\Lambda\left(Y^f(t)\right)\right]+Var\left[\Lambda\left(Y^f(t)\right)\right].
    \end{equation*}
    The derivation of the formula for the covariance function given in \cite{Leon} for the case of $Y_{\alpha}(t)$ is preserved completely for the general case and therefore, we come to the following:
    \begin{equation*}
    cov\left[N\left(Y^f(s)\right),N\left(Y^f(t)\right)\right]=\E\left[\Lambda\left(0,Y^f\left(s\wedge t\right)\right)\right]+cov\left[\Lambda\left(Y^f(s)\right),\Lambda\left(Y^f(t)\right)\right].
    \end{equation*}
    For more details we refer to \cite{Leon}.
\end{remark}

\begin{exam}
    The Poisson process time-changed by the inverse tempered stable subordinator.

    Consider the tempered stable subordinator $H^f(t)$, with the  Bern\v{s}tein function
    \begin{equation}\label{exf}
    f(x)=\left(x+\beta\right)^{\alpha}-\beta^{\alpha}, \quad \alpha\in (0,1), \beta>0.
    \end{equation}
    The corresponding L\'evy measure is given by the formula:
    \begin{equation*}
    \overline{\nu}(ds)=\frac{1}{\Gamma(1-\alpha)}\alpha e^{-\beta s}s^{-\alpha-1}ds;
    \end{equation*}
    and its tail is
    \begin{equation*}
    \nu(s)=\frac{1}{\Gamma(1-\alpha)}\alpha \beta^{\alpha}\Gamma(-\alpha,s),
    \end{equation*}
    where $\Gamma(-\alpha,s)=\int_{s}^{\infty}e^{-z}z^{-\alpha-1}dz$ is the incomplete Camma function.

    The generalized C-D convolution-type derivative \eqref{fDt} for $f(x)$, given by \eqref{exf}, becomes:
    \begin{equation}\label{exfDt}
    ^f{\mathcal D}_t u(t)= \frac{\alpha \beta^{\alpha}}{\Gamma(1-\alpha)}\int_{0}^{t}\frac{\partial }{\partial t }u(t-s)\Gamma(-\alpha,s)ds.
    \end{equation}

    If $N\left(Y^f(t)\right)$ is the Poisson process time-changed by the inverse tempered stable subordinator, that is, $f(x)$ is given by \eqref{exf}, then as consequences of Theorems 2 and 3 we obtain the following governing equations for the  probabilities $p_x^f(t)=\Prob\left[N\left(Y^f(t)\right)=x\right]$:

    \begin{enumerate}
        \item[(i)]
        if $N$ is a non-homogeneous process, then:
        \begin{equation*}
        ^f\mathcal{D}_tp_x^f(t)=\int_{0}^{\infty}\lambda(u)\left[-p_x^f(u)+p_{x-1}^f(u)\right]h^f(t,u)du;
        \end{equation*}

        \item[(ii)]
        if $N$ is a homogeneous process, then:
        \begin{equation*}
        ^f\mathcal{D}_tp_x^f(t)=-\lambda\left[p_x^f(t)-p_{x-1}^f(t)\right];
        \end{equation*}
    \end{enumerate}
    with the same initial conditions as in Theorem 2, where the derivative in the l.h.s. is defined by \eqref{exfDt}, and $h^f(t,u)$ is now the density of the inverse tempered stable subordinator, the exact formula for which is presented, for example, in \cite{AKMS}.

\end{exam}

As an interesting corollary of Theorem 3, we find an eigenfunction of the
convolution derivative of C-D type $^f\mathcal{D}_t$.
\begin{lemma}
    Let Condition I hold. The Laplace transform of the density $l^f(t,u)$ of the inverse subordinator
    \begin{equation*}
    \tilde{l}^f(t,\lambda)=\int_{0}^{\infty}e^{-\lambda u} l^f(t,u)du=\E e^{-\lambda Y^f(t)}
    \end{equation*}
    is an eigenfunction of the generalized C-D convolution-type derivative $^f\mathcal{D}_t$, that is
    \begin{equation}^f\mathcal{D}_t\tilde{l}^f(t,\lambda)=-\lambda \tilde{l}^f(t,\lambda).\label{lapl}\end{equation}
\end{lemma}
\begin{proof}
    The result follows immediately from Theorem 3. Indeed, taking $x=0$ in \eqref{fDpx}, we have:
    $$^f\mathcal{D}_tp_0^f(t)=-\lambda p_0^f(t),$$
    where
    \begin{equation*}
    p_0^f(t)=p_0^f(t,\lambda)=\Prob\left\{N_{\lambda}\left(Y^f(t)\right)=0\right\}=\int_{0}^{\infty}p_0(u) l^f(t,u)du=\int_{0}^{\infty}e^{-\lambda u} l^f(t,u)du=\tilde{l}^f(t,\lambda).
    \end{equation*}
\end{proof}

\begin{remark}
    From Lemma 1 in the case when $Y^f(t)$ is the inverse stable subordinator $Y_{\alpha}(t)$ with $f(s)=s^{\alpha}$, we reveal the well known result that the Laplace transform of the inverse subordinator $Y_{\alpha}(t)$
    $$\E e^{-\lambda Y_{\alpha}(t)}=\mathcal{E}_{\alpha}\left(-\lambda t^{\alpha}\right),$$
    where $\mathcal{E}_{\alpha}\left(\cdotp\right)$ is the Mittag-Leffler function, is an eigenfunction of the C-D fractional derivative:
    $$\frac{\partial^{\alpha}}{\partial t^{\alpha}}\mathcal{E}_{\alpha}\left(-\lambda t^{\alpha}\right)=-\lambda \mathcal{E}_{\alpha}\left(-\lambda t^{\alpha}\right).$$
\end{remark}

\begin{remark}
    The result,
  which we present in Lemma 1 as a direct consequence of Theorem 3,
  has been derived in \cite{K} by means of analytic methods based on the the theory
  of complete Bern\v{s}tein functions.
  It is stated in Theorem 2 \cite{K} that the solution $u(t)$ of the Cauchy problem
    \begin{equation}\label{0}
        \left(\mathbb{D}_{(k)} u\right)(t)=-\lambda u(t), \quad\lambda>0,\,\, t>0,\,\, u(0)=1,
    \end{equation}
    is continuous on $[0,\infty),$ infinitely differentiable and completely monotone,
    under the prescribed conditions on the function $k$,
    which is used to define the differential-convolution operator $\mathbb{D}_{(k)}$:
    \begin{equation}
        \left(\mathbb{D}_{(k)} u\right)(t)=\frac{d}{dt}\int_{0}^{t}k(t-\tau) u(\tau)d\tau-k(t)u(0).
    \end{equation}
    Moreover,
     the probabilistic interpretation of the solution has been given.
     We present here some details to shed more light on properties of the process $N\left(Y^f(t)\right)$.

    Let $H^f$ be a subordinator with the Laplace exponent
    \begin{equation}
        f(x)= bx+\int_{0}^{\infty}\left(1-e^{-xs}\right)\overline{\nu}(ds), b\geq 0,
    \end{equation}
    and either $b>0$, or $\overline{\nu}((0,\infty))=\infty$, or both, and let $Y^f(t)$ be the inverse process defined by \eqref{Yf}.

    The time-changed Poisson process $N(Y^f(t))$ is a renewal process with the waiting  time $J_n$ such that

    \begin{equation}\label{1}
        P\left\{J_n>t\right\}=\E e^{-\lambda Y^f(t)}=\tilde{l}^f(t,\lambda)
    \end{equation}
    and
    \begin{equation}\label{2}
        \int_{0}^{\infty}e^{-st}\E e^{-\lambda  Y^f(t)}dt=\frac{f(s)}{s(\lambda+f(s))}.
    \end{equation}
    (see \cite{MS}, \cite{MNV}, \cite{K}, \cite{T}).

    As can be seen from Theorem 2 from \cite{K}, the expression in the r.h.s. of \eqref{2}
    appears to be the Laplace transform of the solution $u(t)$ to  equation \eqref{0}, which entails
    that the solution itself is given by $\E e^{-\lambda  Y^f(t)}$.

    Therefore, we can see that the result stated in Lemma 1
    is in accordance with the results on the solution to the Cauchy problem \eqref{0} given in \cite{K}.
\end{remark}

\begin{remark}
    The study of arrival times of the time-changed process $N(Y^f(t)$,
    can be done in the same way as that in \cite{Leon} for the case
    of the Poisson process with inverse stable subordinator $N(Y_{\alpha}(t))$.

    We present here some details following to \cite{Leon}.

    The distribution functions for arrival times $T_n$ of homogeneous and
    non-homogeneous Poisson processes are given by
    \begin{equation}\label{1*}
        \mathcal{F}_{T_n}^{HP}(t)=1-e^{-\lambda t}\sum_{k=0}^{n-1}\frac{(\lambda t)^k}{k!}
    \end{equation}
    and
    \begin{equation}\label{2*}
    \mathcal{F}_{T_n}^{NHP}(t)=1-e^{-\Lambda(t)}\sum_{k=0}^{n-1}\frac{(\Lambda(t))^k}{k!}
    \end{equation}
    correspondingly, where for \eqref{2*} to be a distribution function it is required
    for the
     function $\Lambda(t)$ to be a monotone increasing and
    $\Lambda(t) \to 0$, as $t\to 0$,  $\Lambda(t)\to \infty$, as $t\to \infty$ (see, \cite{Leon}).

    The derivation of the expression for distribution of arrival
    times $T_n$ for the process $N(Y^f(t))$ can be done following the same lines of reasoning
    as in \cite{Leon} for the case of time-change by the inverse
    stable subordinator. We obtain the following:
    \begin{enumerate}
        \item[(i)]
        if $N(t)$ is a nonhomegeneons Poisson process, then
\begin{equation*}
        \mathcal{F}_{T_n}^f(t)=\int_{0}^{\infty}l^f(t,u)\mathcal{F}_{T_n}^{NHP}(u)du
    \end{equation*}
        \item[{(ii)}]
        if $N(t)$ is a homogeneous Poisson process, then
        \begin{equation*}
        \mathcal{F}_{T_n}^f(t)=\int_{0}^{\infty}l^f(t,u)\mathcal{F}_{T_n}^{HP}(u)du
        =1-\sum_{k=0}^{n-1} \frac{\lambda^k}{k!}(-1)^k \frac{\partial^k}{\partial\lambda^k}\E e^{-\lambda
        Y^f(t)},
    \end{equation*}
    that is, in the latter case the distribution function of $T_n$ is
    expressed in terms of the Laplace transform of the inverse subordinator $\E e^{-\lambda
        Y^f(t)}$ in the place of the Mittag-Leffler function which
        appears for the particular case of inverse stable subordinator.
    \end{enumerate}

\end{remark}

We are also able to derive the equation for the moment generating function of the homogeneous time-changed Poisson process $N\left(Y^f(t)\right)$.
\begin{lemma}
    The moment generating function of the process $N\left(Y^f(t)\right)$, under the conditions of Theorem 3, is given by the formula
    $$M\left(\theta,t\right)=\tilde{l}^f\left(t,\lambda\left(1- e^{\theta}\right)\right)$$
    and satisfies the differential equation
    \begin{equation}\label{*}
    ^f\mathcal{D}_t M\left(\theta,t\right)=\lambda\left(e^{\theta}-1\right)M\left(\theta,t\right)
    \end{equation}
    with the initial condition $M\left(\theta,0\right)=1$.
\end{lemma}
\begin{proof}
    We can write
    \begin{eqnarray*}
        \E e^{\theta N\left(Y^f(t)\right)}&=&\int_{0}^{\infty}\E e^{\theta N(u)}l^f(t,u)du=\int_{0}^{\infty} e^{-u \left(\lambda-\lambda e^{\theta}\right)}l^f(t,u)du\\
        &=&\tilde{l}^f\left(t,\lambda\left(1- e^{\theta}\right)\right).
    \end{eqnarray*}
    The formula \eqref{*} follows from Lemma 1.
\end{proof}

Reconsidering the proof of Lemma 2, we conclude that the more general result can be stated, namely, for the process $Z(t)=X(Y^f(t))$, where $X(t)$ is an arbitrary L\'{e}vy process.

Indeed, let $X(t)$ be a  L\'{e}vy process with the Laplace exponent $f_X(s)$.

Then we can write:
\begin{eqnarray}
M_Z\left(\theta,t\right)&=&\E e^{\theta Z\left(Y^f(t)\right)}=\int_{0}^{\infty}\E e^{\theta Z(u)}l^f(t,u)du=\int_{0}^{\infty} e^{-u f_X\left(-{\theta}\right)}l^f(t,u)du\nonumber\\
&=&\tilde{l}^f\left(t,f_X\left(-{\theta}\right)\right);
\end{eqnarray}
from Lemma 1 we know that $\tilde{l}^f(t,\lambda)$ is an eigenfunction of the derivative $^f\mathcal{D}_t $, and, therefore,
\begin{equation}
^f\mathcal{D}_t M_Z\left(\theta,t\right)=-f_X\left(-{\theta}\right)M_Z\left(\theta,t\right),
\end{equation}
provided that $f_X\left(-{\theta}\right)$ is well defined.

\section{Skellam processes time-changed by inverse subordinators}

Let $S(t)$ be a Skellam process:
\[
S(t)=N_{1}(t)-N_{2}(t),t\geq 0,
\]
where $N_{1}(t)$ and $N_{2}(t)$ are two independent homogeneous Poisson
processes with intensity parameters $\lambda _{1}>0$ and $\lambda _{2}>0$,
correspondingly.

The probability distribution of $S(t)$ is given by
\begin{equation}
s_{k}(t)=P\left( S(t)=k\right) =e^{-t\left( \lambda _{{1}}+\lambda
    _{2}\right) }\left( \frac{\lambda _{{1}}}{\lambda _{2}}\right)
^{k/2}I_{|k|}\left( 2t\sqrt{\lambda _{{1}}\lambda _{2}}\right) , k\in \Bbb{Z}=\{0,\pm 1,\pm 2,\ldots \},
\label{skellamdistr}
\end{equation}
where $I_{k}$ is the modified Bessel function of the first kind:
\begin{equation*}
I_{k}(z)=\sum_{n=0}^{\infty }\frac{\left( z/2\right) ^{2n+k}}{n!(n+k)!}.
\label{Bessel}
\end{equation*}
The Skellam process is a L\'{e}vy process, its L\'{e}vy measure is
the linear combination of two Dirac measures: $\nu (du)=\lambda
_{1}\delta _{\{1\}}(u)du+\lambda _{2}\delta _{\{-1\}}(du)$, and
the corresponding Bern\v{s}tein function is given by
\[
f_{S}(\theta )=\int_{-\infty }^{\infty }\left( 1-e^{-\theta y}\right) \nu
(dy)=\lambda _{1}\left( 1-e^{-\theta }\right) +\lambda _{2}\left(
1-e^{\theta }\right).
\]

Skellam processes are considered, for example, in \cite{B-N}, the Skellam
distribution had been introduced and studied in \cite{Sk}.

Consider the time-changed Skellam process
\begin{equation}\label{4.1}
Z(t)=S\left(Y^f(t)\right), \quad t\geq0,
\end{equation}
where $Y^f(t)$ is an inverse subordinator.
\begin{theorem}
    Let $Z(t)=S\left(Y^f(t)\right)$, $t\geq0$, be a time-changed Skellam process, where $Y^f(t)$ is the inverse subordinator,  and Condition I holds.

    Then:

    1. The marginal distribution of $Z(t)$ is given by
    $$r_k^f(t)=\Prob\left\{Z(t)=k\right\}=\int_{0}^{\infty}s_k(u) l^f(t,u)du$$
    and satisfies the following system of differential equations:
    \begin{equation}\label{4.2}
    ^f\mathcal{D}_t r_k^f(t)=\lambda_{1}\left(r_{k-1}^f(t)-r_{k}^f(t)\right)-\lambda_{2}\left(r_{k}^f(t)-r_{k+1}^f(t)\right)
    \end{equation}
    with the initial condition
    \begin{equation*}
    r_k(0)=
    \begin{cases}
    1,\quad k=0,\\
    0,\quad k\neq0.\\
    \end{cases}
    \end{equation*}

    2. The moment generating function is given by
    \begin{equation}\label{4.3}
    L\left(\theta,t\right)=\tilde{l}^f\left(t,\left(\lambda_{1}+\lambda_{2}-\lambda_{1}e^{\theta}-\lambda_{2}e^{-\theta}\right)\right),
    \end{equation}
    and satisfies the following differential equation:
    \begin{equation}\label{4.4}
    ^f\mathcal{D}_t L\left(\theta,t\right)=\left[\lambda_{1}\left(e^{\theta}-1\right)+\lambda_{2}\left(e^{-\theta}-1\right)\right]L\left(\theta,t\right)
    \end{equation}
    with the initial condition $L\left(\theta,0\right)=1.$
\end{theorem}
\begin{proof}
    The proof of the first part is analogous to that of Theorem 3.
    Using conditioning arguments, we obtain the formula for the marginal distribution of $Z(t)=S\left(Y^f(t)\right)$:
    $$r_k^f(t)=\int_{0}^{\infty}s_k(u) l^f(t,u)du,$$
    where $s_k(u)$ are given by \eqref{skellamdistr}.

    We take the derivative $^f\mathbb{D}_t$ and use formula \eqref{dlf}:
    \begin{equation}\label{4.5}
    ^f{\mathbb{D}}_t r_k^f(t)
    =- \int_{0}^{\infty}s_k(u)\frac{\partial}{\partial  u}l^f(t,u)du
    =\int_{0}^{\infty}l^f(t,u)\frac{\partial}{\partial  u}s_k(u)du-s_k(u)l^f(t,u)\big|_{0}^\infty;
    \end{equation}
    Since the probabilities $s_k(u)$ satisfy the equation
    \begin{equation*}
    \frac{\partial}{\partial  u}s_k(u)du=\lambda_{1}\left(s_{k-1}(t)-s_k(t)\right)-\lambda_{2}\left(s_{k}(t)-s_{k+1}(t)\right),
    \end{equation*}
    from \eqref{4.5} we obtain:
    \begin{equation}\label{4.6}
    ^f\mathbb{D}_t r_k^f(t)=\lambda_{1}\left(r_{k-1}^f(t)-r_{k}^f(t)\right)-\lambda_{2}\left(r_{k}^f(t)-r_{k+1}^f(t)\right)+s_k(0)l^f(t,0).
    \end{equation}
    Using relation \eqref{fDtrelation} between the derivatives $^f\mathcal{D}_t$ and $^f\mathbb{D}_t$, we can write
    \begin{equation}\label{4.7}
    ^f\mathcal{D}_t r_k^f(t)=^f\mathbb{D}_tr_k^f(t)-\nu (t)r_k^f(0),
    \end{equation}
    we also have (see \eqref{dlfin}):
    \begin{equation}\label{4.8}
    l^f(t,0)=\nu(t)
    \end{equation}
    and
    \begin{equation}\label{4.9}
    r_k^f(0)
    =- \int_{0}^{\infty}s_k(u)l^f(0,u)du
    =\int_{0}^{\infty}s_k(u)\delta(u)du=s_k(0).
    \end{equation}

    In view of \eqref{4.7}-\eqref{4.9}, from \eqref{4.6} we obtain equation \eqref{4.2}.

    For the moment generating function we can write:
    \begin{eqnarray*}
        \E e^{\theta Z(t)}&=&\E e^{\theta S\left(Y^f(t)\right)}=\int_{0}^{\infty}\E e^{\theta Z(u)}l^f(t,u)du=\int_{0}^{\infty} e^{-u \left(\lambda_{1}+\lambda_{2}-\lambda_{1}e^{\theta}-\lambda_{2}e^{-\theta}\right)}l^f(t,u)du\\
        &=&\tilde{l}^f\left(t,\left(\lambda_{1}+\lambda_{2}-\lambda_{1}e^{\theta}-\lambda_{2}e^{-\theta}\right)\right)
    \end{eqnarray*}
    and, therefore, formula \eqref{4.3} is obtained. The equation \eqref{4.4} follows in view of Lemma 1.
\end{proof}

\begin{remark}
    In the case when $Y^f(t)$ is an inverse stable subordinator, from Theorem 5 we obtain the known results for so-called fractional Skellam process $S\left(Y_{\alpha}(t)\right)$, stated in \cite{KLS}(Theorem 3.2), that is, equations \eqref{4.2} and \eqref{4.4} hold with the fractional C-D derivative in the l.h.s.
\end{remark}

\section{Concluding remarks}
In the paper the time-changed processes $N(Y^f(t)), \,\,t\ge 0,
\,\, \mbox{\rm and} \,\, S(Y^f(t)), \,\,t\ge 0, $ are studied,
where $N(t)$ is a (non-homogeneous) Poisson process, $S(t)$ is a
Skellam process, and $Y^f(t)$ is the inverse process for the
subordinator with Bern\v{s}tein function $f$. The governing
equations for marginal distributions of the processes $N(Y^f(t))$
and $S(Y^f(t))$ are presented in terms of the generalised
Caputo-Djrbashian convolution-type derivatives with respect to the
function $f$ defined in \cite{K}, \cite{T}. For the case where
$Y^f(t)$ is an inverse stable subordinator, the obtained results
coincide with the known results  for fractional (in time) Poisson
and Skellam processes.

\section*{Acknowledgments}
The authors are grateful to the referee for valuable comments and
suggestions which helped to improve the paper.

%%%%%%%%%%%%%%%%%%%%%%%%%%%%%%%%%%%%%%%%%%%%%%%%


\begin{thebibliography}{1}
\bibitem{ALM}
G. Aletti, N.N. Leonenko, and E. Merzbach, \textit{ Fractional
Poisson fields and martingales}, Journal of Statistical Physics
\textbf{170} (2018), no. 4, 700-730.

\bibitem{AKMS} M.S. Alrawashdeh, J.F. Kelly, M.M. Meerschaert, H.-P. Scheffler, \textit{Applications of inverse tempered stable subordinators}, Computers ans Mathematics with Applications (2016).

\bibitem{B-N}  O. E. Barndorff-Nielsen, D. Pollard and N. Shephard,
\textit{Integer-valued Levy processes and low latency financial econometrics}, Quant.
Finance \textbf{12} (2011), no. 4, 587--605.

\bibitem{BO09} L. Beghin, E. Orsingher, \textit{Fractional Poisson processes and related planar random motions}, Electron J. Probab. \textbf{14} (2009), no. 61, 1790--1827.

\bibitem{BO10} L. Beghin, E. Orsingher, \emph{Poisson-type processes governed by fractional and higher-order recursive differential equations}, Electron J. Probab. \textbf{15} (2010), no. 22, 684--709.

\bibitem{B}  J. Bertoin, \textit{L\'{e}vy Processes}, Cambridge University
Press, Cambridge, (1996).

\bibitem{BS} K. Buchak, L. Sakhno, Compositions of Poisson and
    Gamma processes, {\it Modern Stochastics: Theory and
        Applications}, {\bf 4}(2), 161--188 (2017).

\bibitem{GOS}    R. Garra, E. Orsingher, M. Scavino, \textit{Some probabilistic properties of fractional point processes}, Stochastic Analysis and Application \textbf{35} (2017), no. 4, 701--718.

\bibitem{KLS}  A. Kerss, N. Leonenko, A. Sikorskii, \textit{Fractional Skellam
    processes with applications to finance}, Fractional Calculus and Applied
Analysis \textbf{17} (2014), no 2, 532--551.

\bibitem{KS}
    K. Kobylych and L. Sakhno, \emph{
Point processes subordinated to compound Poisson processes},
Theory of Probability and Mathematical Statistics \textbf{94}
(2016), 85--92  (in Ukrainian);English translation  in Theory of
Probability and Mathematical Statistics \textbf{94},(2017), 89-96.

\bibitem{K}
 A. N. Kochubei, \textit{General fractional calculus, evolution equations, and renewal processes}, Integral Equations Operator Theory \textbf{71} (2011), no. 4, 583--600.

\bibitem{KNV}
A. Kumar, E. Nane and P. Vellaisamy, \textit{Time-Changed Poisson
Processes}, Stat. Probab. Lett. \textbf{81}(12), 1899--1910
(2011).

\bibitem{Leon} N. Leonenko, E. Scalas, M. Trinh, \textit{The fractional non-homogeneous  Poisson process.} Statist. Probab. Lett. \textbf{120}  (2017), 147--156.


\bibitem{MS}
     M.M. Meerschaert and H.-P. Scheffler, \textit{Triangular array limits for continuous time random wakrs},  Stoch. Proc. Appl. \textbf{118}  (2008), 1606--1633; \textbf{120} (2010), 2520--2521.

 \bibitem{MNV}
     M.M. Meerschaert, E. Nane, and P. Vellaisamy, \textit{ The fractional Poisson process and the inverse stable subordinator}, Electronic Journal of Probability, \textbf{16}  (2011), Paper no. 59, 1600--1620.

\bibitem{OP}  E. Orsingher and F. Polito, \textit{The space-fractional Poisson
    process}, Statistics and Probability Letters \textbf{82} (2012),  852--858.

\bibitem{OT}  E. Orsingher and B. Toaldo, \textit{Counting processes with
    Bernstein intertimes and random jumps}, Journal of Applied Probability \textbf{52}, (2015),
1028--1044.

\bibitem{Sato}  K. Sato, \textit{L$\acute{e}$vy processes and infinitely
    divisible distributions}, Cambridge University Press (1999).

\bibitem{Sk}  J. G. Skellam, \textit{The frequency distribution of the
    difference between two Poisson variables belonging to different populations}, J. of the Royal Statistical Society, Ser. A (1946), 109--296.

\bibitem{T}  B. Toaldo. \textit{Convolution-type derivatives, hitting-times of subordinators and time-changed $C_0$-semigroups}, Potential Analysis  \textbf{42}, (2015), 115--140.

\end{thebibliography}
\end{document}